\documentclass[12pt]{article}   	
\usepackage{geometry}                		
\usepackage{graphicx}				
\usepackage{amssymb}
\usepackage{mathtools}
\usepackage{enumerate}
\usepackage{tikz}
\usepackage{amsmath}
\usepackage{textcomp}
\usepackage[utf8]{inputenc}
\usepackage[english]{babel}
\usepackage[scr]{rsfso}
\usepackage{mathtools}
\usepackage{amsthm}
\usepackage{hyperref}
\usepackage{comment}
\usepackage{ytableau}
\usepackage{youngtab}
\usepackage{rotating}
\usepackage{forest}

\usepackage{bbm}
\usepackage{url}

\usepackage{lipsum}

\newtheorem{thm}{Theorem}
\numberwithin{thm}{section}
\newtheorem{lem}[thm]{Lemma}

\newtheorem{prop}[thm]{Proposition}
\newtheorem{rem}[thm]{Remark}
\theoremstyle{definition}
\newtheorem{defi}[thm]{Definition}
\newtheorem{eg}[thm]{Example}

\usetikzlibrary{cd}

\newcommand{\n}{\null}
\renewcommand{\b}{\mbox{$\bullet$}}

\renewcommand{\l}{\lambda}
\newcommand{\e}{\emptyset}
\newcommand{\cover}{\lessdot}
\newcommand{\rd}{\mathrm{R2}\delta_n}
\newcommand{\rk}{\mathrm{R}k\delta_n}
\newcommand{\rkk}{\mathrm{R}(k+1)\delta_n}
\newcommand{\ro}{\mathrm{R\Lambda}_n}
\newcommand{\ds}{2\delta_n}
\newcommand{\ks}{k\delta_n}
\newcommand{\dn}{\delta_n}
\newcommand{\bso}{\mathrm{bso}}
\renewcommand{\P}{\Phi_n}
\newcommand{\Pk}{\Phi^{(k)}_n}
\newcommand{\Tn}{\mathcal{T}_{n}}
\newcommand{\T}{\mathcal{T}_{n+1}}
\newcommand{\Tkn}{\mathcal{T}^{(k)}_{n}}
\newcommand{\Tk}{\mathcal{T}^{(k)}_{n+1}}

\newcommand{\Tkkn}{\mathcal{T}^{(k+1)}_{n}}
\renewcommand{\S}{\mathfrak{S}_n}
\renewcommand{\SS}{\mathfrak{S}_{n+1}}

\newcommand{\Proj}{\Pi_n}

\newcommand{\p}{p_n(x)}

\newcommand{\rdl}{\mathrm{R_{L}2}\delta_n}
\newcommand{\rdr}{\mathrm{R_{R}2}\delta_n}

\newcommand{\bor}{\mathrm{bor}}

\renewcommand{\ne}{\nearrow}
\newcommand{\se}{\searrow}

\newcommand{\D}{\mathcal{D}}

\newcommand{\s}{\sigma}

\newcommand{\prof}{\nabla(\s)}
\newcommand{\profp}{\nabla(\s')}
\newcommand{\profs}[1]{\nabla(\s^{(#1)})}

\newcommand{\proft}[1]{\nabla(\tau^{#1})}
\newcommand{\proftlist}{(\proft{(n-1)},\ldots, \proft{(0)})}

\newcommand{\f}{f_\s}
\newcommand{\fp}{f_{\s'}}

\title{Bijection between Increasing Binary Trees and Rook Placements on Double Staircases}
\author{Bishal Deb\footnote{bishal.deb.19@ucl.ac.uk}\footnote{Dept. of Mathematics, University College London}}
\date{January 13, 2023}

\begin{document}

\maketitle

\begin{abstract}
In this paper, we shall construct a bijection between rook placements on double staircases (introduced by Josuat-Verg\`es in 2017) and increasing binary trees. We introduce two subclasses of rook placements on double staircases, which we call left and right-aligned rook placements. We show that their enumeration, while keeping track of a certain statistic, gives the $\gamma$-vectors of the Eulerian polynomials. We conclude with a discussion on a different bijection that fits in very well with our main bijection, and another discussion on generalising our main bijection. Our main bijection is a special case of a bijection due to Tewari (2019).
\end{abstract}

\section{Introduction}

For a given $n\in\mathbb{N}_{>0}$, consider the partition $(2n, 2(n-1),\ldots ,2)$ and let $\ds$ denote its Young diagram in French notation. We call this diagram the \emph{double staircase} on $n$ rows. By a \emph{rook placement on $\ds$}, we mean a placement of $n$ non-attacking rooks in $\ds$ such that each row has exactly one rook and each column has at most one rook. Rook placements on $\ds$ were studied by Josuat-Verg\`es in relation to Stammering Tableaux in \cite{stamtab2017}.

In this paper, we shall state a bijection between rook placements on $\ds$ and increasing binary trees on $n+1$ vertices. Our bijection is rather simple, and is, in fact, a special case of a bijection constructed by Tewari in \cite{Tewari2019}. Tewari's bijection is between maximal rook placements on the board $(n-1)\times kn$ and labelled rooted plane $k$-ary trees. We obtain our bijection by taking $k=2$ and restricting the bijection to increasing trees. However, our exposition of this bijection is quite different from that of Tewari's, and we state it in a manner more suited for our applications.

We then introduce two subclasses of rook placements on $\ds$, which we call \emph{left-aligned} and \emph{right-aligned rook placements}. As an application of our bijection, we will show that the left-aligned and right-aligned rook placements on $\ds$ are equidistributed with respect to a certain statistic and are, in fact, given by the $\gamma$-vector of the $(n+1)^{\text{st}}$ Eulerian polynomial. The study of left-aligned and right-aligned rook placements was motivated by looking at a ``projection'' map from stammering tableaux to oscillating tableaux (first introduced as up-down tableaux in \cite{Sun86}) and trying to understand the image and fibres of this map.

Let us now state our two main theorems before moving over to the more technical parts of the paper; for this, we introduce a few definitions and notations. 

Let $\rd$ denote the set of all rook placements on $\ds$, and let $\Tn$ denote the set of all increasing binary trees on $n$ vertices. 

\begin{defi} Given a Young diagram in French notation, we define a \emph{block} to be a maximal collection of columns with the same height, and the size of a block is its number of columns. By blocks of a rook placement, we shall refer to the blocks of its underlying Young diagram.
\end{defi}

Notice that for a rook placement $R\in\rd$, all blocks of $R$ are of size $2$. We shall index the blocks in $R$ from left to right.

\begin{thm}
There is a bijection $\P$ between $\rd$, the set of all rook placements on the double staircase $\ds$, and the set $\T$, the set of all increasing binary trees on $n+1$ vertices with labels from the set $\{1,\ldots,n+1\}$, in which the following correspondence between properties of the $i^{\text{th}}$ block in $\P(R)$ in $R\in \rd$ and the vertex $i$ in $\P(R)$ hold.

\begin{tabular}{l|l}
\hline
  For a rook placement $R$ on $\ds$, &  For an increasing binary tree $\P(R)$ on \\
  the $i^{\text{th}}$ block in $R$ has & $n+1$ vertices,   vertex $i$ in $\P(R)$ has\\ 
\hline
\hline
 -- rooks in both columns & -- two children\\
 -- no rook in either column & -- no children \\
 -- a rook in left column but no rook & -- a left child but no right child \\
in right column &\\
 -- a rook in right column but no rook & -- a right child but no left child \\
in left column &\\
 \hline
\end{tabular}
\label{thm.mainbij}
\end{thm}

We say that a rook placement $R\in \rd$ is \emph{left-aligned} if none of its blocks contains a rook in its right column without having a rook in its left column. Similarly, we say that a rook placement $R\in \rd$ is \emph{right-aligned} if none of its blocks contains a rook in its left column without having a rook in its right column.
Note that a rook placement may be both left and right aligned if each pair of columns forming a block contains zero or two rooks.
Let $\rdl$ be the subset of $\rd$ consisting of left-aligned rook placements, and let $\rdr$ be the subset of $\rd$ consisting of right-aligned rook placements. 

Given $R\in \rd$, let $\bor(R)$ denote the number of blocks with one rook in $R$. It is not difficult to see that the number of left-aligned rook placements on $\ds$ with $k$ blocks having exactly one rook is equal to the number of right-aligned rook placements on $\ds$ with $k$ blocks having exactly one rook. With this in mind, we can state our second main theorem which is an application of Theorem~\ref{thm.mainbij}.

\begin{thm} Let $\p$ be the polynomial
\begin{equation}
\p \coloneqq \sum_{T\in \rdl} x^{\bor(T)} =\sum_{T\in \rdr} x^{\bor(T)}.
\end{equation}

\begin{itemize}
\item[\textup{(a)}] The coefficient of $x^{n-2k}$ in $\p$ is the number of permutations on $n+1$ letters with $k$ descents such that every descent is a peak. This is the also same as the number of $n+1$ letters with $k$ ascents such that every ascent is a valley.

\item[\textup{(b)}] The coefficients of $x^{n-2k}$ are equal to the $\gamma$-vectors of the Eulerian polynomial and the following identity holds:
\begin{equation}
x^{n/2}p_n\left(\dfrac{x+1}{\sqrt{x}}\right) = S_{n+1}(x)
\end{equation}
where $S_n(x)$ is the $n^{\text{th}}$ Eulerian polynomial (see Definition~\ref{defi.eulerian}).
\end{itemize}
\label{thm.mainapp}
\end{thm}

We provide motivation for our work while stating the necessary results and introducing notation in Section~\ref{sec.prelim}. We then state our main bijection and prove our two main theorems in Section~\ref{sec.proofs}. This is followed by some discussion in Section~\ref{sec.comments}, firstly on an alternate bijective proof of Theorem~\ref{thm.mainapp}, and then on generalising our main bijection in Theorem~\ref{thm.mainbij}.

\section{Preliminaries}\label{sec.prelim}

Given a partition $\l = (\l_1,\l_2,\ldots,\l_k)$ where $\l_1\geq \l_2\geq \cdots \geq \l_k$, its \emph{Young diagram in French notation} is a diagramatic representation of $\l$ drawn as follows:

We put a row of $\l_1$ squared cells, just above it we put a row of $\l_2$ squared cells which begin just above the first cell of the previous row and so on.

Throughout this paper, whenever we say Young diagram of a partition $\l$, we shall refer to its Young diagram in French notation, unless mentioned otherwise.

Now we define oscillating tableaux (first introduced as up-down tableaux in \cite{Sun86}) and stammering tableaux (introduced in \cite{stamtab2017}) from which this paper was motivated.

\begin{defi} An \emph{oscillating tableaux} of size $n$ is a finite sequence of partitions $(\l^{(0)},\ldots,\l^{(n)})$ where $\l^{(0)} = \l^{(n)} = \e$ such that $\l^{(i)} \cover \l^{(i+1)}$ or $\l^{(i+1)}\cover \l^{(i)}$. Here $\cover$ is the covering relation on the Young's lattice, i.e., $\l \lessdot \mu$ if and only if the Young diagram of $\mu$ contains the Young diagram of $\l$ and has exactly one extra cell. 
\end{defi}

\begin{eg}
\begin{equation}
\left(\e,\yng(1),\yng(2),\yng(1),\yng(1,1),\yng(1),\e \right)\end{equation}
is an example of an oscillating tableau of size $6$.
\label{eg.osc}	
\end{eg}

\begin{defi} A \emph{stammering tableaux} of size $n$ is a finite sequence of partitions $(\l^{(0)},\ldots,\l^{(3n)})$ where $\l^{(0)} = \l^{(3n)} = \e$ such that 
\begin{itemize}
\item if $i\equiv 0 \text{ or } 1\pmod{3}$ then either $\l^{(i)} \cover \l^{(i+1)}$ or $\l^{(i)} = \l^{(i+1)}$,
\item if $i\equiv 2\pmod{3}$ then $\l^{(i+1)} \cover \l^{(i)}$.
\end{itemize}
\end{defi}

\begin{eg}
\begin{equation}
\left(\e,\yng(1),\yng(2);\yng(1),\yng(1,1),\yng(1,1);\yng(1),\yng(1),\yng(1);\e\right)\end{equation}
is an example of a stammering tableau of size $3$.
\label{eg.stam}
\end{eg}

Given a stammering tableau, one can obtain an oscillating tableau by removing multiple consecutive occurrences of the same Young diagrams. For example, the oscillating tableau in Example~\ref{eg.osc} can be obtained from Example~\ref{eg.stam} by removing the multiple ocurrences of the partitions $(1,1)$ and $(1)$.

We want to determine the image and fibres of this map from the set of stammering tableaux to the set of oscillating tableaux. Roby \cite{Roby91} and Krattenthaler \cite{Krattenthaler2016} used Fomin's growth diagrams \cite{Fomin1988} to show that oscillating tableaux correspond to rook placements on some Young diagrams; Josuat-Verg\`es \cite{stamtab2017} also used Fomin's growth diagrams to show that stammering tableaux correspond to rook placements on $\ds$. Using their correspondences, we translate our problem of determining the image and fibres of this map to the setting of rook placements in Young diagrams, in which it becomes easier to answer. 

A \emph{rook placement in a Young diagram} of a partition $\l$ is a placement of rooks in the cells of the Young diagram such that no two rooks are in the same row or column. Let $\Lambda_n$ denote the set of Young diagrams with $n$ rows and $n$ columns and let $\ro$ denote the set of rook placements in Young diagrams in $\Lambda_n$ such that each row and column has exactly one rook. Also, recall that $\rd$ is the set of rook placements on $\ds$ with exactly one rook in each row and at most one rook in each column.

The following proposition acts as a translation from oscillating and stammering tableaux to rook placements via Fomin's local rules:

\begin{prop}
\begin{enumerate}
\item \textup{\cite[Theorem~2]{Krattenthaler2016}} The set of all oscillating tableaux of size $n$ is in bijection with $\ro$.

\item \textup{\cite[Section~2]{stamtab2017}} The set of all stammering tableaux of size $n$ is in bijection with $\rd$.

\end{enumerate}

\end{prop}

In fact, due to \cite[Theorem~2]{Krattenthaler2016} our original projection map can be translated into the following:

Let $R\in \rd$  be a rook placement on the double staircase $\ds$. Let $\Proj: \rd\to \ro$ be the map such that $\Proj(R)$ is the rook placement obtained by removing all empty columns of $R$. 

We are interested in determining the image of $\Proj$ and its fibres.

\begin{eg} An example of  a rook placement $R$ in $2\delta_5$ and the corresponding oscillating tableaux $\Pi_5(R)$ is drawn in Figure~\ref{fig.projrook}.

\end{eg}

\begin{figure}
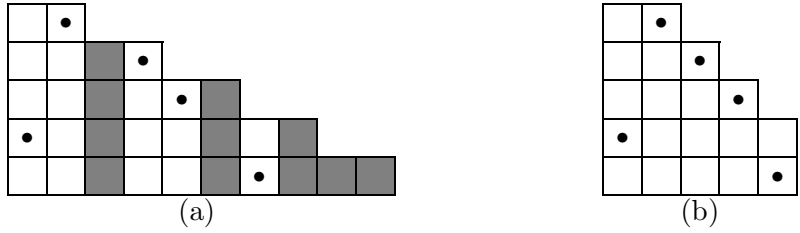

\centering
\begin{minipage}{.4\textwidth}
\centering
\scalebox{0.8}{
\begin{ytableau}
\n & \b\\
\n & \n & *(gray)\n & \b\\
\n & \n & *(gray)\n & \n & \b & *(gray)\n\\
\b & \n & *(gray)\n & \n & \n & *(gray)\n & \n & *(gray)\n\\
\n & \n & *(gray)\n & \n & \n & *(gray)\n & \b & *(gray)\n & *(gray)\n & *(gray)\n\\
\end{ytableau}}\\
\small(a)
\end{minipage}%
\begin{minipage}{.5\textwidth}
\centering
\scalebox{0.8}{
\begin{ytableau}
\n & \b\\
\n & \n  & \b\\
\n & \n &  \n & \b \\
\b & \n &  \n & \n &  \n \\
\n & \n &  \n & \n  & \b  \\
\end{ytableau}
}\\
\small(b)
\end{minipage}
\caption{Figure~(a) is a rook placement $R$ on $2\delta_5$ and Figure~(b) its projection $\Pi_5(R)$. The empty columns of $R$ are shaded in gray.}
\label{fig.projrook}
\end{figure}

Notice that, for $R\in\rd$, all blocks of $R$ are of size $2$, and the blocks of $\Proj(R)$ are of size at most $2$. Also, given $S\in \ro$ such that the maximum size of a block is $2$, we can construct $R\in \rd$ such that $\Proj(R) = S$.

We can now determine the image and fibres of the projection map as follows:

\begin{prop} The image $\Proj(\rd)$ is the set of all rook placements $S\in\ro$ such that each block is of size at most $2$.
\label{prop.projimage}
\end{prop}

\begin{proof} Given $R\in \rd$, notice that all blocks of $\Proj(R)$ are of size at most $2$. 

Conversely, let $S\in \ro$ be a rook placement in which the blocks are of size at most $2$. It is not difficult to see that there is a unique left-aligned rook placement $R$ on $\ds$ such that $\Proj(R) = S$. In fact, $R$ can be constructed as follows:
\begin{itemize}
\item[\textup{(a)}] if $S$ does not contain a block of height $i$, then the block of height $i$ in $R$ has no rook in it,
\item[\textup{(b)}] if $S$ contains a block of height $i$ of size $1$, then the left column of the block of height $i$ in $R$ is a copy of the block in $S$ and no rook is inserted in the right column,
\item[\textup{(c)}] if $S$ contains a block of height $i$ of size $2$, then the block of height $i$ in $R$ is a copy of the block in $S$.
\end{itemize}

Thus, the image of $\Proj$ is exactly the set of all rook placements $S\in\ro$ such that each block is of size at most $2$.
\end{proof}

Notice that in the proof of Proposition~\ref{prop.projimage}, we could have constructed a right-aligned rook placement instead of a left aligned-rook placement by interchanging left and right in case~\textup{(b)}.

Also, for $S\in\Proj(\rd)$ let $\bso(S)$ denote the number of blocks of size $1$ in $S$. Then it is clear that the fibre $\Proj^{-1}(S)$ has cardinality $2^{\bso(S)}$. This is because one can construct a rook placement $R\in \rd$ such that $\Proj(R) = S$ by following the cases in the proof of Proposition~\ref{prop.projimage}, except in case~\textup{(b)} one can choose to insert the copy of the block of size $1$ either as the left column or the right column in $R$.

As the cardinality of $\rd$ is $(n+1)!$ (\cite[Proposition~7.2]{stamtab2017}), we get that
\begin{equation}
\sum_{S\in \Proj(\rd)} 2^{\bso(S)} = (n+1)!.
\end{equation}

We can replace $2$ with $x$ in the left hand side expression to obtain:
\begin{equation}
\sum_{S\in \Proj(\rd)} x^{\bso(S)} = \sum_{R\in \rdl} x^{\bor(R)} = \sum_{R\in \rdr} x^{\bor(R)} = \p.
\label{eq.pndef}
\end{equation}

The first few polynomials $p_n$ are given in Table~\ref{tab.p_n}. We observe that the coefficients of $\p$ seem to be the same as the triangle of numbers \cite[A101280]{Oeis} which are the $\gamma$-vectors of the Eulerian numbers (see~\cite[Chapter~4]{Petersen2015}). This leads us to Theorem~\ref{thm.mainapp}.

\begin{table}[h!]
\centering
\caption{First few $p_n$.}\label{tab.p_n}
\begin{tabular}{c | l } 
 \hline
 $n$ & $p_n(x)$  \\ [0.5ex] 
 \hline\hline
 $1$ & $x$  \\ 
 $2$ & $x^2 + 2$  \\
 $3$ & $x^3 + 8x$ \\
 $4$ & $x^4 + 22x^2 + 16$ \\
 $5$ & $x^5 + 52x^3 + 136x$ \\ 
 $6$ & $x^6 + 114x^4 + 720x^2 + 272$ \\
 $7$ & $x^7 + 240x^5 + 3072x^3 + 3968x$\\
 $8$ & $x^8 + 494x^6 + 11616x^4 + 34304x^2 + 7936$ \\
 $9$ & $x^9 + 1004x^7 + 40776x^5 + 230144x^3 + 176896x$ \\
 $10$ & $x^{10} + 2026x^8 + 136384x^6 + 1328336x^4 + 2265344x^2 + 353792$ \\
 $11$ & $x^{11} + 4072x^9 + 441568x^7 + 6949952x^5 + 21953408x^3 + 11184128$\\[1ex] 
 \hline
\end{tabular}
\end{table}

Another class of important objects in this paper are increasing binary trees. Let $[n]\coloneqq \{1,\ldots,n\}$. An \emph{increasing binary tree} on $n$ vertices is a rooted tree with $n$ vertices each having a unique label from the set $[n]$ such that each vertex has one or zero left child and one or zero right child, and the label of each vertex is smaller than the label of its parent. Let $\Tn$ denote the set of all increasing binary trees on $n$ vertices.

\begin{eg} Figure~\ref{fig.inctree} is an increasing binary tree on $6$ vertices.
\end{eg}

\begin{figure}
 \centering 
 \scalebox{0.8}{
 \begin{forest}for tree={circle,draw, l sep=20pt,calign=fixed angles}[1,black,edge = black,[5],[2,[,phantom],[3,[4,[6],[,phantom]],[,phantom]]]]\end{forest} }
\caption{An increasing binary tree on $6$ vertices.}
\label{fig.inctree}
\end{figure}
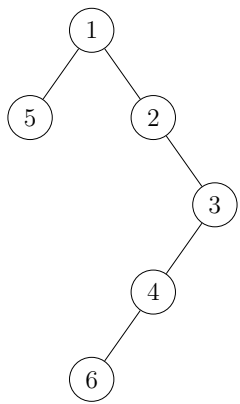

Also if $\s=\s_1\ldots \s_n \in \S$ is a permutation on $n$ letters, then under the assumption that $\s_0 = \s_{n+1} = 0$, any $i\in[n]$ can be one of the following:
\begin{itemize}
\item $i$ is said to be a \emph{valley} if $\s_{i-1}>\s_i$ and $\s_i<\s_{i+1}$,
\item $i$ is said to be a \emph{double fall} if $\s_{i-1}>\s_i>\s_{i+1}$,
\item $i$ is said to be a \emph{double rise} if $\s_{i-1}<\s_i<\s_{i+1}$,
\item $i$ is said to be a \emph{peak} if $\s_{i-1}<\s_i$ and $\s_i>\s_{i+1}$.
\end{itemize}

In \cite[Section~1.5, pp.~44-45]{EC1}, Stanley describes a well known bijection between $\S$ and $\Tn$ such that given any $\s\in \S $, the bijection gives an increasing binary tree $T(\s)$ on $n$ vertices in which for each $i\in[n]$ there is a correspondence between properties of position $i$ in $\s$ and vertex $i$ in $T(\s)$ as given in Table~\ref{tab.propcorrespondence}.

\begin{eg} For $\s = 512643$, $T(\s)$ is the increasing binary tree in Figure~\ref{fig.inctree}.

\end{eg}

\begin{table}
\begin{center}
\begin{tabular}{l|l}
\hline
Position $i$ in $w$, where $\s\in S$  & Vertex $i$ in $T(\s)$ where $T(\s)\in \Tn$ \\
is a permutation on $n$ letters & is an increasing binary tree on $n$ vertices\\
\hline
\hline
-- valley & -- has two children\\
-- peak & -- has no children\\
-- double fall & -- has a left child but no right child\\
-- double rise & -- has a right child but no left child\\
\hline
\end{tabular} 
\end{center}
\caption{Correspondence between properties of $i\in [n]$ in a permutation $\s\in \S$ and the increasing binary tree obtained via the bijection $\s\to T(\s)$ in \cite[Section~1.5, pp.~44-45]{EC1}.}
\label{tab.propcorrespondence}
\end{table}

Also, let us now define the Eulerian polynomials and their $\gamma$-vectors.

\begin{defi} Given a permutation $\s=\s_1\ldots \s_n \in \S$, an index $i\in[n-1]$ is said to be a \emph{descent} if $\s_i>\s_{i+1}$, i.e., $i$ is a double fall or a peak; $i$ is said to be an \emph{ascent} if $\s_i<\s_{i+1}$ if $i$ is either a double rise or a valley. Let $\mathrm{des}(\s)$ denote the number of descents of $w$ and let $\mathrm{asc}(\s)$ denote the number ascents of $\s$. 
\label{defi.descentascent}
\end{defi}

\begin{defi} The $n^{\mathrm{th}}$ Eulerian polynomial is defined as 
\[S_n(x) \coloneqq \sum_{\s\in \S} x^{\mathrm{des}(\s)}.\]
\label{defi.eulerian}
\end{defi}

It is not difficult to see that the number descents of a permutation is equal to the number of ascents of its reversal. Thus, we also have that,
\[S_n(x) = \sum_{\s\in \S} x^{\mathrm{asc}(\s)}.\]

Notice that the $(n+1)^{\mathrm{st}}$ Eulerian polynomial is of degree $n$ and hence, $S_{n+1}(x) \in \mathbb{C}\langle 1,x,\ldots,x^n\rangle$, the vector space over $\mathbb{C}$ consisting of polynomials of degree less than or equal to $n$. If we write $S_{n+1}(x)$ in the basis $\{(x+1)^n,x(x+1)^{n-2},x^2(x+1)^{n-4},\ldots\}$ then we obtain coefficients $\gamma_{n,0},\gamma_{n,1},\gamma_{n,2},\ldots$ such that
\begin{equation}
S_{n+1}(x) = \gamma_{n,0} (x+1)^n + \gamma_{n,1} x(x+1)^{n-2} + \gamma_{n,3} x^{2}(x+1)^{n-4}+\ldots.
\label{eq.gammavecdef}
\end{equation}

Foata and Sch\"utzenberger \cite{Foata1970} showed that the Eulerian polynomials are $\gamma$-positive, i.e., the coefficients $\gamma_{n,0},\gamma_{n,1},\gamma_{n,2},\ldots$ are all non-negative. Furthermore, they also showed that $\gamma_{n,k}$ is the number of permutations on $n+1$ letters with $k$ descents such that every descent is a peak which is the same as the number of permutations on $n+1$ letters with $k$ ascents such that every ascent is a valley. See \cite[Chapter~4]{Petersen2015} for more on $\gamma$-vectors of Eulerian polynomials.

\section{Proof of Main Results}\label{sec.proofs}

\subsection{The Bijection and Proof of Theorem~\ref{thm.mainbij}}\label{subsec.mainbij}

Throughout this section, the rows and columns of $2\delta_n$ are numbered from top to bottom and from left to right when drawn in French notation, i.e., the first row has $2$ cells, the second row has $4$ cells and so on, and the first and second columns have $n$ cells, the third and fourth columns have $n-1$ cells and so on. 

Let $R\in \rd$ be a rook placement on the double staircase $2\delta_n$. We construct a rooted binary tree $\P(R)$ on $n+1$ vertices as follows:

\begin{itemize}

\item The root gets label $1$.

\item If the left column (right column) of the $i^{\text{th}}$ block has a rook in its $j^{\text{th}}$ row, the vertex $i$ in the tree gets a left child (right child resp.) and it is labelled $j+1$.

\item If the left column (right column) of the $i^{\text{th}}$ block has no rook then the vertex $i$ in the tree gets no left child (right child resp.). 

\end{itemize}

Now we show that $\P$ gives us increasing binary trees and that this map is a bijection.

\begin{lem} Given $R\in \rd$, $\P(R)$ is an increasing binary tree.

\end{lem}

\begin{proof}  According to the construction of $\P(R)$, each vertex $i$ has a set of left children and a set of right children; but, as there is at most one rook in each column of $R$, $i$ has zero or one left child and zero or one right child. Thus, $\P(R)$ is a forest of binary trees.
If $j$ is a child of $i$ in $\P(R)$, then $i+1\leq j\leq n+1$ as the $i^{\text{th}}$ block in $R$ only has rows $i$ to $n$, and hence, $i<j$. Thus, $\P(R)$ is a forest of increasing binary trees.
Finally, as the row $j$ of $T$ has exactly one rook which is in the $i^{\text{th}}$ block, all vertices in $\{2,\ldots,n+1\}$ have a parent. Thus, $\P(R)$ is an increasing binary tree.
\end{proof}

\begin{proof}[Proof of Theorem~\ref{thm.mainbij}] We show that the map $\P$ is a bijection that satisfies the required properties.

Let $T\in \T$ be an increasing binary tree on $n+1$ vertices. We can construct $\Psi_n(T)$, a rook placement on $\ds$ as follows:

\begin{itemize}

\item If $j+1$ is the left child of $i$ in $T$, we put a rook in the $j^{\text{th}}$ row of the left column of the $i^{\text{th}}$ block.

\item If $j+1$ is the right child of $i$ in $T$, we put a rook in the $j^{\text{th}}$ row of the right column of the $i^{\text{th}}$ block.

\end{itemize}

As each vertex $j+1$ has a parent in $T$, each row $j$ in $\Psi_n(T)$ will have exactly one rook. As each parent $i$ in $F$ has at most one left child and at most one right child, each column of $\Psi_n(T)$ will have at most one rook. Thus, $\Psi_n(T)$ is a rook placement on $2\delta_n$.

From the construction, it is clear that $\P$ and $\Psi_n$ are inverses of each other. Also, notice that vertex $i$ in $T$ has a left child if and only if the left column in the $i^{\text{th}}$ block of $\Psi_n(T)$ has a rook. Similarly, vertex $i$ in $T$ has a right child if and only if the right column in the $i^{\text{th}}$ block of $\Psi_n(T)$ has a rook. Thus, the correspondences between properties of $i$ hold.
\end{proof}

\begin{eg} If $R$ is the rook placement in Figure~\ref{fig.projrook}(a), then the increasing binary tree $\Phi_5(R)$ is the tree drawn in Figure~\ref{fig.inctree}.

\end{eg}

\subsection{Proof of Theorem~\ref{thm.mainapp}}

\begin{lem} Let $R\in \rd$ be a rook placement on $\ds$. The number of empty blocks of $R$ is the same as the number of blocks of $R$ with exactly $2$ rooks in them.

\end{lem}

\begin{proof}
Let $i$ be the number of blocks of $R$ with exactly $1$ rook. There are $n-i$ rooks left to be placed and as each of the remaining blocks with rooks in them have $2$ rooks, there are $(n-i)/2$ with $2$ rooks. All the remaining blocks have no rooks in them and as there are $n$ blocks, we get that there are $n-(n-i)/2 - i=(n-i)/2$ blocks with no rooks  in them.
\end{proof}

\begin{proof}[Proof of Theorem~\ref{thm.mainapp}] 
\begin{itemize}
\item[\textup{(a)}] 
Let $\p = a_{n,0}x^n + a_{n,1}x^{n-2}+\ldots$. From definition of $\p$, we know that $a_{n,k}$ is the number of left-aligned rook placements $R\in \rdl$ such that $R$ has $n-2k$ blocks with exactly $1$ rook, and  $S\in\rdr$ such that $S$ has $n-2k$ blocks with exactly $1$ rook. Notice that $R$ has exactly $k$ blocks with no rooks and $S$ has exactly $k$ blocks with $2$ rooks. 

From Theorem~\ref{thm.mainbij} we get that $a_{n,k}$ is equal to the number of increasing binary trees on $n+1$ vertices in which no vertex has a right child without having a left child, and there are exactly $k$ vertices with no children. Similarly, we also get that $a_{n,k}$ is equal to the number of increasing binary trees in which no vertex has a left child without having a right child, and there are exactly $k$ vertices with $2$ children. Thus, from Table~\ref{tab.propcorrespondence} we get that $a_{n,k}$ is the number of permutations on $n+1$ letters in which there are no double falls and $k$ peaks, i.e., each descent is a peak and there are $k$ descents. Similarly, we also get that $a_{n,k}$ is the number of permutations on $n+1$ letters in which there are no double rises and $k$ valleys, i.e., each ascent is a valley and there are $k$ ascents.

\item[\textup{(b)}] Notice that
\begin{equation}
x^{n/2}p_n(\dfrac{x+1}{\sqrt{x}}) = a_{n,0} (x+1)^n + a_{n,1} x(x+1)^{n-2} + a_{n,3} x^{2}(x+1)^{n-4}+\ldots.
\label{eq.pntransformcoefficients}
\end{equation}

In part~(a), we have established that $a_{n,k} = \gamma_{n,k}$ where $(\gamma_{n,0},\gamma_{n,1},\ldots)$ is the $\gamma$-vector of the Eulerian polynomial $S_{n+1}(x)$ as described in Equation~\eqref{eq.gammavecdef}. Thus, by comparing Equations~\eqref{eq.gammavecdef} and~\eqref{eq.pntransformcoefficients}, we get the desired result.
\end{itemize}
\end{proof}

\section{Further Comments}\label{sec.comments}

\subsection{Another Bijective Proof of Theorem~\ref{thm.mainapp} via Chains of Dyck Shapes}

In the previous sections, we constructed a bijection between rook placements on double staircases and permutations via increasing binary trees. In this section, we shall construct a bijection via a different intermediate object called chains of Dyck shapes, which we then use to prove Theorem~\ref{thm.mainapp}.

\subsubsection{Chains of Dyck Shapes and Rook Placements on Double Staircases}

Chains of Dyck shapes were introduced by Josuat-Verg\'es in~\cite[Definition~3.7]{stamtab2017} where a bijection to \emph{subdivided Laguerre histories} was also constructed \cite[Section~6]{stamtab2017}. However, our approach will be slightly different and our motivation comes from \emph{large Laguerre histories} of Fran\c{c}on-Viennot \cite{Francon1979} (also see \cite[Chapter~4b]{ViennotIMSc2016}).

Before proceeding further, let us set some notation on Dyck paths and words that encode Dyck paths.

\begin{defi} A \emph{Dyck word} of size $n$ is a word $w$ of length $2n$ from the alphabet $\{\ne,\se\}$ such that
\begin{itemize}
\item the number of occurrences of $\ne$ is equal to the number of occurrences of $\se$ .
\item for each prefix of $w$, the number of occurrences of $\ne$ is greater than the number of occurrences of $\se$.
\end{itemize}

\end{defi}

A \emph{Dyck path} is a diagrammatic representation of a Dyck word; it is a path on the positive quadrant which begins at $(0,0)$ and ends at $(2n,0)$, such that whenever $\ne$ occurs we take a $(1,1)$ step, and whenever $\se$ occurs we take a $(1,-1)$ step.

Given $R\in \rd$, a rook placement on $\ds$, we define $d(R)$ to be the word of length $2n+2$ such that 
\begin{itemize}
\item the first letter is $\ne$, the last letter is $\se$,
\item if $2\leq i\leq 2n+1$, the $i^{\mathrm{th}}$ letter is $\ne $ if the $(i-1)^{\mathrm{st}}$ column contains  a rook and $\se$ otherwise.
\end{itemize}
Josuat-Verg\'es showed that $d(R)$ is a Dyck word (\cite[Lemma~3.2]{stamtab2017}).

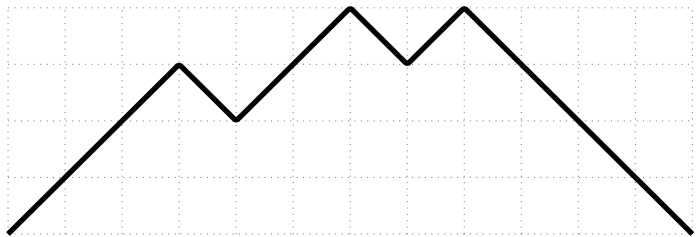
\begin{figure}\centering
\begin{tikzpicture}[scale = 0.75]
\draw[gray,dotted] (-1,-1) grid (11,3);
\draw[rounded corners=1, color=black, line width=2] (-1,-1)-- (0,0)--(2,2)--(3,1)--(5,3)--(6,2)--(7,3)--(10,0)--(11,-1);
\end{tikzpicture}\caption{Dyck path for the Dyck word $\ne\ne\ne\se\ne\ne\se\ne\se\se\se\se$.}\label{fig.dyckpath}
\end{figure}

\sloppy \begin{eg} If $R$ is the rook placement in Figure~\ref{fig.projrook}(a) then the Dyck path $d(R)$ is drawn in Figure~\ref{fig.dyckpath} and it is given by the Dyck word $\ne\ne\ne\se\ne\ne\se\ne\se\se\se\se$.
\end{eg}

Let $\dn$ denote the partition $(n,n-1,\ldots,1)$. (We call it a single staircase.)
\begin{defi} A \emph{Dyck shape} of size $n$ is a skew shape $\dn /\l$ such that $\l \leq \delta_{n-1}$. 

Here $\leq$ denotes the partial order on the Young lattice, i.e., for partitions $\l$ and $\mu$, $\l\leq \mu$ if the Young diagram of $\l$ is contained in the Young diagram of $\mu$.

\end{defi} 

Dyck shapes are in bijection with Dyck paths which is clearer when they are drawn in Japanese notation ($135^{\circ}$ clockwise rotation to the French notation).

\begin{eg} Figure~\ref{fig.dyckshape} is the diagrammatic representation of the Dyck shape $\delta_6/(3,1)$ and it corresponds to the Dyck word $\ne\ne\ne\se\ne\ne\se\ne\se\se\se\se$.

\end{eg}

\begin{figure}
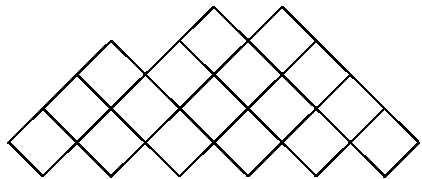

\centering
\begin{turn}{-135}
\begin{ytableau}
\n\\
\n & \n \\
\n & \n & \n\\
\n & \n & \n & \n \\
\none & \n & \n & \n & \n\\
\none & \none & \none & \n & \n & \n\\
\end{ytableau}
\end{turn}
\vspace{-5em}
\caption{Dyck shape $\delta_6/(3,1)$.}\label{fig.dyckshape}
\end{figure}

\begin{defi} A skew shape is called a \emph{ribbon} if its Young diagram is connected and contains no $2\times 2$ square. 

Let $D$ and $E$ be two Dyck shapes of size $n$ and $n+1$, respectively. We draw the skew diagrams of $D$ and $E$ in Japanese notation such that the leftmost cells of $D,E$ coincide. If $D$ is contained inside $E$, and the difference $E/D$ is a ribbon then we write $D\sqsubset E$.

\end{defi}

\begin{eg} The Dyck shape~$D = \delta_5/(2)$ and the Dyck shape~$E = \delta_6/(3,1)$ are such that~$D\sqsubset E$ as illustrated in Figure~\ref{fig.dyckribbon}.
\end{eg}

\begin{rem}\label{rem.sqsubset} Notice that this definition is easily translated in terms of binary words over~$\ne$ and~$\se$. Let~$D$ be a Dyck word, then~$D \sqsubset E$ if and only if~$E$ is obtained from~$D\se\se$ by changing a~$\se$ into a~$\ne$ (and
each step~$\se$ of~$D\se\se$ can be changed except the last one, so that there are~$n + 1$ possibilities for a path of length~$2n$).
\end{rem}

\begin{figure}
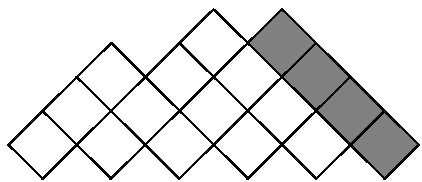

\centering
\begin{turn}{-135}
\begin{ytableau}
*(gray)\n\\
*(gray)\n & \n \\
*(gray)\n & \n & \n\\
*(gray)\n & \n & \n & \n \\
\none & \n & \n & \n & \n\\
\none & \none & \none & \n & \n & \n\\
\end{ytableau}
\end{turn}
\vspace{-2cm}
\caption{Here $D\sqsubset E$ where $D = \delta_5/(2)$ and $E = \delta_6/(3,1)$. Here the gray cells denote the ribbon.}\label{fig.dyckribbon}
\end{figure}

\begin{defi} An \emph{$n$-chain of Dyck shapes} is a sequence $D_1\sqsubset D_2 \sqsubset \ldots \sqsubset D_n$ such that $D_i$ is a Dyck shape of size $i$. The biggest path $D_n$ is called the shape of the chain and we say that the chain ends at $D_n$.

\end{defi}

\begin{eg} The rook placement~$R$ in Figure~\ref{fig.projrook}~(a) corresponds to the chain of Dyck shapes in Figure~\ref{fig.dyckchain}.
\end{eg}

\begin{figure}
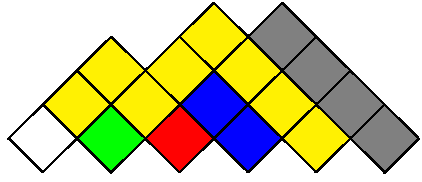

\centering
\begin{turn}{-135}
\begin{ytableau}
*(gray)\n\\
*(gray)\n & *(yellow)\n \\
*(gray)\n & *(yellow)\n & *(blue)\n\\
*(gray)\n & *(yellow)\n & *(blue)\n & *(red)\n \\
\none & *(yellow)\n & *(yellow)\n & *(yellow)\n & *(green)\n\\
\none & \none & \none & *(yellow)\n & *(yellow)\n & \n\\
\end{ytableau}
\end{turn}
\vspace{-2cm}
\caption{The rook placement $R$ in Figure \ref{fig.projrook}(a) corresponds to this chain of Dyck shapes. Here the coloured cells denote the different ribbons.}\label{fig.dyckchain}
\end{figure}

We use $\D_n$ to denote the set of $n$-chains of Dyck shapes. Josuat-Verg\'es showed that $|\D_{n+1}| = (n+1)!$ by constructing a bijection between $D_{n+1}$ and $\rd$ (\cite[Proposition~3.8]{stamtab2017}). We shall state this bijection now without proof.

\begin{thm}{\cite[Proposition~3.8]{stamtab2017}} Let $R\in \rd$, a rook placement on $\ds$. Let $R_i$ be the rook placement in $2\delta_i$ obtained by keeping only the top $i$ rows. Then 
\begin{itemize}
\item[(a)] $(d(R_0), \ldots, d(R_n))$ corresponds to an $(n+1)$-chain of Dyck shapes. Here $R_0$ is the unique empty rook placement and $d(R_0) = \ne \se$. 

\item[(b)] The map $R\to (d(R_0),\ldots,d(R_n))$ defines a bijection between rook placements on $\ds$ and $(n+1)$-chains of Dyck shapes.
\end{itemize}
\label{thm.bijchainrook}
\end{thm}

\subsubsection{Permutations and Chains of Dyck Shapes}

Let us define the large profile of a permutation. Let $\s \in \S$ and let $\s_0 = \s_{n+1} = 0$. We define the function $\f:[n-1]\to \{\ne,\se\}$ as follows: 
\begin{itemize}
\item if $i$ is a valley then $\f(i) = \ne \ne$,
\item if $i$ is a peak then $\f(i) = \se \se$,
\item if $i$ is a  double rise then $\f(i) = \se \ne$,
\item if $i$ is a double fall then $\f(i) = \ne \se$.
\end{itemize}

Let $\prof$ be the word $\prof = \ne \f(1)\cdots\f(n-1)\se$. We can show that this word is a Dyck word. If the values of $\f$ were instead $\ne\se$ for a double rise, and $\se\ne$ for a double fall, then the Dyck word obtained would give the Dyck path given in \cite[Section~3]{Francon1979}. We call it the \emph{profile} or \emph{large Laguerre profile} of $\s$.

\begin{eg} Let $\s = 512643$, then $1$ is a valley; $5,6$ are peaks; $2$ is a double rise; and $3,4$ are double falls. Thus, $\prof = \ne \ne\ne \se\ne \ne\se \ne\se \se\se \se$, the corresponding Dyck path is the one in Figure~\ref{fig.dyckpath}, and the corresponding Dyck shape is the one in Figure~\ref{fig.dyckshape}.\end{eg}

For $\s\in\S$, let $\s'$ be the permutation obtained by removing the letter $n$ from $\s$ in one-line notation. Let $\s^{(0)} = \s$ and inductively define $\s^{(i+1)} = (\s^{(i)})'$.

We have the following lemma:
\begin{lem} For $\s\in \SS$, $\profp \sqsubset\prof$. Thus, the tuple of Dyck shapes given by $(\profs{(n)},\ldots,\profs{(0)})$ forms an $(n+1)$-chain of Dyck shapes. 
\label{lem.shapetupleischain}
\end{lem}

\begin{proof} To prove this lemma, we need to observe how we obtain $\prof$ from $\profp$ when we insert $n+1$  into $\s'$. We have the following cases:

\emph{Case 1}: When $n+1$ is the first letter of $\s$, and let $i$ be the second letter of $\s$, i.e., the first letter of $\s'$. Then, $i$ is either a double rise or a peak in $\s'$ and hence the first letter of $\fp(i)$   is $\se$. In $\s$, $i$ becomes a double fall if it was a peak in $\s'$, and a valley if it was a double rise in $\s'$. Thus, the first letter of $\f(i)$ is $\ne$, and the second letter is the same as the second letter of $\fp(i)$. Thus, $\prof$ is obtained from $\profp\se \se$ by changing the $(2i)^{\mathrm{th}}$ letter from $\se$ to $\ne$.

\emph{Case 2}: When $n+1$ is the last letter of $\s$, and let $i$ be the last letter of $\s'$. It is either a peak or a double fall in $\s'$ and it becomes a double rise or a valley, respectively, in $\s$. In either case, the second letter of $\f(i)$ and $\fp(i)$ are different, while the first letters are equal. Thus, $\prof$ is obtained from $\profp\se\se$ by changing the $(1+2i)^{\mathrm{th}}$ letter from $\se$ to $\ne$.

\emph{Case 3}: When $n+1$ is inserted between $i$ and $j$ where $i<j$. Thus $\f(j) = \fp(j)$. If $n+1$ is to the left of $i$, then $i$ is either a double rise or a peak in $\s'$ and we use arguments similar to those in Case~1 to show that $\prof$ is obtained from $\profp\se\se$ by chainging the $(2i)^{\mathrm{th}}$ letter from $\se$ to $\ne$. Similarly, if $n+1$ is to the right of $i$, then it is either a double fall or a valley in $\s'$ and we can proceed by similar arguments to those in Case~2 to show that $\prof$ can be obtained from $\profp\se\se$ by changing the $(1+2i)^{\mathrm{th}}$ letter from $\se$ to $\ne$.

Thus, using Remark~\ref{rem.sqsubset} in all three cases we get that $\profp \sqsubset\prof$.

\end{proof}

\begin{thm} For $\s\in\SS$, the map $\s\mapsto (\profs{(n)}, \ldots, \profs{(0)})$ forms a bijection between permutations and $(n+1)$-chains of Dyck shapes.
\label{thm.bijdyckshape}
\end{thm}

\sloppy \begin{proof} We inductively construct the reverse bijection for this map. Given $(D_0,\ldots, D_n)$ an $(n+1)$-chain of Dyck shapes, let $\tau$ be the permutation given the $n$-chain $(D_0,\ldots,D_{n-1})$ such that $(D_0,\ldots, D_{n-1}) = \proftlist$.

Thus, we are done if we construct $\s$ such that $\s' = \tau$ and $\prof = D_n$. From Remark~\ref{rem.sqsubset}, we know that $D_n$ is obtained from $D_{n-1}\se\se$ by replacing one of the $\se$ with $\ne$. From the proof of Lemma~\ref{lem.shapetupleischain}, we know that changing one of the $\se$ to $\ne$ corresponds uniquely to each of the $n+1$ different ways of inserting the letter $n+1$ to $\tau$. We can state this explicitly as follows:

\begin{itemize}
\item If $D_{n-1}\se\se$ and $D_n$ differ at the $(2i)^{\mathrm{th}}$ position, $\s$ is obtained by inserting $n+1$ to the left of $i$ in $\tau$.
\item If $D_{n-1}\se\se$ differ at the $(1+2i)^{\mathrm{th}}$ position $\s$ is obtained by inserting $n+1$ to the right of $i$ in $\tau$.
\end{itemize}
\end{proof}

\begin{proof}[Second Proof of Theorem~\ref{thm.mainapp}(a)] From Theorem~\ref{thm.bijdyckshape} and Theorem~\ref{thm.bijchainrook} (\cite[Proposition~3.8]{stamtab2017}) the set $\rd$ is in bijection with $\D_{n+1}$ which is in bijection with $\SS$. If for $R\in \rd$, $\s\in \SS$ is the permutation that $R$ corresponds to under this bijection, then we get that the correspondence in Table~\ref{tab.rookdyckperm} holds.

\begin{table}
\centering
\begin{tabular}{l|c|l}
\hline
Block $i$ in $R$ & Letters $(2i)$ and $(2i+1)$ in $d(R)$ & $i$ in $\s$\\
\hline
\hline
-- rook in both columns & -- $\ne \ne$ & -- valley\\
-- no rooks & -- $\se \se$ &  -- peak\\
-- rook in left column but  & -- $\se \ne$ & -- double rise\\
no rook in right column & & \\
-- rook in right column but & -- $\ne \se$ & -- double fall\\
no rook in left column & & \\
\hline
\end{tabular}
\caption{Correspondence between properties of position $i\in[n]$ in a permutation $\s\in\SS$, block $i$ in a rook placement $R\in \rd$, and letters $f(2i)$ and $f(2i+1)$ of shape of the corresponding to the $(n+1)$-chain of Dyck shapes under bijections $R\to (d(R_0),\ldots,d(R_n))$ and reverse of $\s\mapsto (\profs{(n)}, \ldots, \profs{(0)})$, respectively.}\label{tab.rookdyckperm}
\end{table}

Thus, we get that the left-aligned rook placement on $\ds$ with $\bor(T) = k$ correspond to permutations on $n+1$ letters such that there are $k$ descents and every descent is a peak.

Similarly, we get that the right-aligned rook placements on $\ds$ with $\bor(T) = k$ correspond to permutations on $n+1$ letters such that there are $k$ ascents and every ascent is a valley.
\end{proof}

We have seen several bijections so far. In fact, four of these bijections satisfy the following commutative diagram in Equation~\eqref{eq.commdiag} holds. We can compare the table in Theorem~\ref{thm.mainbij} and Tables~\ref{tab.propcorrespondence} and~\ref{tab.rookdyckperm} to see the correspondence between $i\in[n]$ for the four different objects involved.
\begin{equation}
\begin{tikzcd}
\SS\arrow[rr,rightharpoonup,yshift=0.25ex]\arrow[rr,leftharpoondown,yshift=-0.25ex]\arrow[dd,rightharpoonup,xshift=0.25ex]\arrow[dd,leftharpoondown,xshift=-0.25ex]& & \D_{n+1} \arrow[dd,rightharpoonup,xshift=0.25ex]\arrow[dd,leftharpoondown,xshift=-0.25ex]\\
& \circlearrowleft  &\\[-2ex]
\T\arrow[rr,rightharpoonup,yshift=0.25ex]\arrow[rr,leftharpoondown,yshift=-0.25ex]& & \rd\\
\end{tikzcd}\label{eq.commdiag}
\end{equation}

\subsection{Higher Arity and Generalisations}

The discussion in this subsection is just an overview of what might, perhaps, become the subject of another paper. 

In Section \ref{subsec.mainbij}, we constructed a bijection between $\rd$, the set of rook placements on $\ds$, and $\T$, the set of increasing binary trees on $n+1$ vertices. We shall now look at some further generalisations of this bijection.

For $k\in \mathbb{N}_{>0}$, let $\ks$ denote the Young diagram $(kn,k(n-1),\ldots, k)$ and we call these $k$-tuple staircases. Let $\rk$ denote the set of rook placements on $\ks$ with exactly one rook in each row and at most one rook in each column. Also, let $\Tkn$ denote the set of increasing $k$-ary trees on $n$ vertices, i.e., vertex-labelled trees in which each internal vertex has zero or one first child, zero or one second child, and so on, and the label of each vertex is greater than the label of its parent. We can construct a bijection $\Pk:\rk\to \Tk$ such that for any $T\in \rk$, $\Pk(T)$ is given as follows:

\begin{itemize}
\item $1$ is a root,
\item for $1\leq j \leq k$, if the $j^{\mathrm{th}}$ column of the $u^{\mathrm{th}}$ block is empty, the vertex $u$ has no $j^{\mathrm{th}}$ child. If the $j^{\mathrm{th}}$ column of the $u^{\mathrm{th}}$ block has a rook in row $v$, the vertex $u$ has  $v+1$ as its $j^{\mathrm{th}}$ child.
\end{itemize}
It is not difficult to see that this map is indeed a bijection between $\rk$ and $\Tk$, and has a fairly simple inverse bijection. 

A bijection between $k$-Stirling permutations on the multiset $\{1^k,\ldots,n^k\}$ (permutations on the multiset $\{1^k,\ldots,n^k\}$ avoiding the pattern $2$-$1$-$2$) and $\Tkkn$ was mentioned by Park in \cite{Park1994} who attributes it to Gessel. The details of this bijection were first described by Janson, Kuba, and Panholzer in \cite[Theorem 1]{Janson2011}. Thus, this generalises Theorem~\ref{thm.mainbij} by establishing a bijection between $k$-Stirling permutations on the multiset $\{1^k,\ldots,(n+1)^k\}$, increasing $(k+1)$-ary trees on $(n+1)$ vertices and rook placements on $\rkk$. We can also track several statistics of interest in these bijections the details of which might be stated out in a different paper.

We can go even further and to even more generality. Given a tuple $s = (s_1,\ldots,s_n)\in \mathbb{N}_{>0}^{n}$, let $\mathfrak{S}_{s}$ denote the set of generalised $s$-Stirling permutations, i.e., permutations on the multiset $\{1^{s_1},\ldots,n^{s_k}\}$ avoiding the pattern $2$-$1$-$2$. Also, let $\mathcal{T}_{s}$ denote the set of increasing $s$-trees, which are increasing labelled trees on the vertex set $[n]$ such that vertex $i$ has zero or one first child, zero or one second child, $\ldots$, zero or one $s_i^{\mathrm{th}}$ child. Kuba and Panholzer in \cite[Theorem~1]{KubaPanholzer2011} have constructed a bijection between the sets $\mathfrak{S}_s$ and $\mathcal{T}_{s+\mathbbm{1}}$.

Now, coming back to rook placements, consider the shape given by the partition $(s_1+\ldots+s_n,\ldots, s_{n-1}+s_n,s_n)$ and let us denote it by $\delta_s$. Let $\mathrm{R}\delta_s$ denote the set of rook placements on $\delta_s$ such that there is exactly one rook in each row and at most one rook in each column.  Let $t = (s_1,\ldots, s_n,m)$ for some $m\in \mathbb{N}_{>0}$. We can construct a bijection between the sets $\mathrm{R}\delta_s$ and $\mathcal{T}_{t}$ as follows:
\begin{itemize}
\item $1$ is a root,
\item for $1\leq j \leq s_u$, if the $j^{\mathrm{th}}$ column of the $u^{\mathrm{th}}$ block is empty, the vertex $u$ has no $j^{\mathrm{th}}$ child. If the $j^{\mathrm{th}}$ column has a rook in row $v$, the vertex $u$ has  $v+1$ as its $j^{\mathrm{th}}$ child.
\end{itemize}

We are yet to find the complete picture to be able to make a commutative diagram analogous to the one in Equation~\eqref{eq.commdiag}. To be able to do this we intend to find an analogous construction for profiles of generalised Stirling permutations and perhaps extend the bijection by Fran\c{c}on and Viennot in \cite{Francon1979}.

\section{Acknowledgements}

Most of this work was done during my masters at Universit\'e Paris-Est Marne-la-Vall\'ee (now Universit\'e Gustave Eiffel) in 2018-19.  My thesis advisors were Samuele~Giraudo and Matthieu~Josuat-Verg\`es. I would also like to thank Jean-Christophe~Novelli and Matthieu~Fradelizi for being the examiners in my thesis defence and for their helpful comments on my thesis. 

The remaining part of this work was mostly done while preparing for an online talk at the seminar series run by Manjil~P.~Saikia and Parama Dutta (\href{https://manjilsaikia.in/seminar/}{https://manjilsaikia.in/seminar/}), and I thank them for inviting me as a speaker. I thank Vasu Tewari for useful correspondence, A.R.~Balasubramanian for advice on structuring this paper, and the anonymous reviewers for their helpful remarks. Finally, I would like to thank Xavier~Viennot for teaching me how to think bijectively.

I was financially supported by Labex B\'ezout (\href{http://bezout.univ-paris-est.fr/}{http://bezout.univ-paris-est.fr/}) during my masters.

\bibliographystyle{plainurl}
\bibliography{references}

\end{document}